\documentclass[10pt,leqno]{amsart}
\usepackage{indentfirst, amssymb,amsmath,amsthm}
\newtheorem*{theoA}{Theorem A}
\newtheorem*{theoB}{Theorem B}
\newtheorem*{theoC}{Theorem C}
\newtheorem*{theoD}{Theorem D}
\newtheorem*{theoE}{Theorem E}
\newtheorem*{theoF}{Theorem F}
\newtheorem*{theoG}{Theorem G}
\newtheorem*{theoH}{Theorem H}
\newtheorem*{theoI}{Theorem I}
\newtheorem{theo}{Theorem}[section]
\newtheorem{lem}{Lemma}[section]

\newtheorem{exm}{Example}[section]
\newtheorem{defi}{Definition}[section]
\newtheorem{rem}{Remark}[section]
\newtheorem{ques}{Question}[section]
\newcommand{\ol}{\overline}
\newcommand{\be}{\begin{equation}}
\newcommand{\ee}{\end{equation}}
\newcommand{\beas}{\begin{eqnarray*}}
\newcommand{\eeas}{\end{eqnarray*}}
\newcommand{\bea}{\begin{eqnarray}}
\newcommand{\eea}{\end{eqnarray}}
\newcommand{\lra}{\longrightarrow}
\numberwithin{equation}{section}
\AtBeginDocument{{\noindent\small Commun. Korean Math. Soc. 31 (2016), No. 2, pp. 311–327\\
DOI: 10.4134/CKMS.2016.31.2.311}}
\begin{document}
\title[ On the Generalizations Of Br\"{U}ck Conjecture]{On the Generalizations Of Br\"{U}ck Conjecture}
\date{}
\author[A. Banerjee and B. Chakraborty ]{ Abhijit Banerjee * and Bikash Chakraborty }
\date{}
\address{ Department of Mathematics, University of Kalyani, West Bengal 741235, India.}
\email{abanerjee\_kal@yahoo.co.in, abanerjee\_kal@rediffmail.com, abanerjeekal@gmail.com.}
\email{bikashchakraborty.math@yahoo.com, bikashchakrabortyy@gmail.com.}
\maketitle
\let\thefootnote\relax
\footnotetext{2000 Mathematics Subject Classification: 30D35.}
\footnotetext{Key words and phrases: Meromorphic function, derivative, small function.}
\footnotetext{Type set by \AmS -\LaTeX}
\setcounter{footnote}{0}
\begin{abstract} We obtain similar types of conclusions as that of Br\"{u}ck \cite{3} for two differential polynomials which in turn radically improve and generalize several existing results.
Moreover a number of examples have been exhibited to justify the necessity or sharpness of some conditions used in the paper. At last we pose an open problem for future research.\end{abstract}
\section{Introduction Definitions and Results}
Let $f$ and $g$  be two non constant meromorphic functions in the open complex plane $\mathbb{C}$.

If for some $a\in\mathbb{C}\cup\{\infty\}$, $f$ and $g$ have same set of $a$-points with the same multiplicities, we say that $f$ and $g$ share the value $a$ CM (counting multiplicities) and if we do not consider the multiplicities then $f$, $g$ are said to share the value $a$ IM (ignoring multiplicities).
 When $a=\infty$ the zeros of $f-a$ means the poles of $f$. Let $m$ be a positive integer or infinity and $a\in\mathbb{C}\cup\{\infty\}$.
We denote by $E_{m)}(a;f)$ the set of all $a$-points of $f$ with multiplicities not exceeding $m$, where an $a$-point is counted according to its multiplicity.
Also we denote by $\ol E_{m)}(a;f)$ the set of distinct $a$-points of $f(z)$ with multiplicities not greater than $m$.
If for some $a\in\mathbb{C}\cup\{\infty\}$, $E_{m)}(a,f)=E_{m)}(a,g)$ ($\ol E_{m)}(a,f)=\ol E_{m}(a,g)$) holds for $m=\infty$ we say that $f$, $g$ share the value $a$ CM (IM).

It will be convenient to let $E$ denote any set of positive real numbers of finite linear measure, not necessarily the same at each occurrence.
For any non-constant meromorphic function $f$, we denote by $S(r, f)$ any quantity satisfying $$S(r, f) = o(T(r, f))\;\;\;\;\;\;\;\;\;\;\ (r\lra \infty, r\not\in E).$$

A meromorphic function $a(\not\equiv \infty)$ is called a small function with respect to $f$ provided that $T(r,a)=S(r,f)$ as $r\lra \infty, r\not\in E$.
If $a=a(z)$ is a small function we define that $f$ and $g$ share $a$ IM or $a$ CM according as $f-a$ and $g-a$ share $0$ IM or $0$ CM respectively.

We use $I$ to denote any set of infinite linear measure of $0<r<\infty$.\\
Also it is known to us that the hyper order $\rho _{2}(f)$ of $f(z)$ is defined by \beas \rho _{2}(f)=\limsup\limits_{r\lra\infty}\frac{ \log \log T(r,f)}{\log r}.\eeas
Nevanlinna's uniqueness theorem shows that two meromorphic functions $f$ and $g$ share $5$ values IM are identical. Rubel and Yang \cite{10a} first showed for entire functions that in the special situation where $g$ is the derivative of $f$, one usually needs sharing of only two values CM for their uniqueness.
 Two years later, Mues and Steinmetz \cite{10} proved that actually in the above case one does not even need the multiplicities. They proved the following  result :
\begin {theoA}\cite{10} Let $f$ be a non-constant entire function. If $f$ and $f^{'}$ share two distinct values $a$, $b$ IM then $f^{'}\equiv f$.\end{theoA}
Subsequently, there were more generalizations with respect to higher derivatives as well.

Natural question would be to investigate the relation between an entire function and its derivative counterpart for one CM shared value. In 1996, in this direction the following famous conjecture was proposed by Br\"{u}ck \cite{3}:\\
{\it {\bf Conjecture:} Let $f$ be a non-constant entire function such that the hyper order $\rho _{2}(f)$ of $f$ is not a positive integer or infinite. If $f$ and $f^{'}$ share a finite value $a$ CM, then $\frac{f^{'}-a}{f-a}=c$, where $c$ is a non zero constant.} \par
Br\"{u}ck himself proved the conjecture for $a=0$. For $a\not=0$, Br\"{u}ck \cite{3} showed that under the assumption $N(r,0;f^{'})=S(r,f)$ the conjecture was true without any growth condition when $a=1$.
\begin {theoB}\cite{3} Let $f$ be a non-constant entire function. If $f$ and $f^{'}$ share the value $1$ CM and if $N(r,0;f^{'})=S(r,f)$ then $\frac{f^{'}-1}{f-1}$ is a nonzero constant.\end {theoB}
Following example shows the fact that one can not simply replace the value $1$ by a small function $a(z)(\not\equiv 0,\infty)$.
\begin{exm}\label{ex1.1} Let $f=1+{e^{e^{z}}}$ and $a(z)=\frac{1}{1-e^{-z}}$. \end{exm} By {\it Lemma 2.6} of \cite {4} [p. 50] we know that $a$ is a small function of $f$. Also it can be easily seen that $f$ and $f^{'}$ share $a$ CM and $N(r,0;f^{'})=0$ but $f-a\not=c\;(f^{'}-a)$ for every nonzero constant $c$. We note that $f-a=e^{-z}\;(f^{'}-a)$.
So in this case additional suppositions are required. \par
However for entire function of finite order, Yang \cite{11} removed the supposition $N(r,0;f^{'})=0$ and obtained the following result.
\begin {theoC}\cite{11} Let $f$ be a non-constant entire function of finite order and let $a(\not=0)$ be a finite constant. If $f$, $f^{(k)}$ share the value $a$ CM then $\frac{f^{(k)}-a}{f-a}$ is a nonzero constant, where $k(\geq 1)$ is an integer.\end{theoC}
{\it Theorem C} may be considered as a solution to the Br\"{u}ck conjecture.
Next we consider the following examples which show that in {\it Theorem B} one can not simultaneously replace ``CM" by ``IM" and ``entire function" by ``meromorphic function".
\begin{exm}\label{ex1.2}$f(z)=1+tanz$. \end{exm}
Clearly $f(z)-1=tanz$ and $f^{'}(z)-1=tan^{2}z$ share $1$ IM and $N(r,0;f^{'})=0$.
\begin{exm}\label{ex1.3} $f(z)=\frac{2}{1-e^{-2z}}$. \end{exm}
Clearly $f^{'}(z)=-\frac{4e^{-2z}}{(1-e^{-2z})^{2}}$. Here $f-1=\frac{1+e^{-2z}}{1-e^{-2z}}$ and $f{'}-1=-\frac{(1+e^{-2z})^{2}}{(1-e^{-2z})^{2}}$.\\
Here $N(r,0;f^{'})=0$ So in both the examples we see that the conclusion of {\it Theorem B} ceases to hold.\par
From the above discussion it is natural to ask the following question.
\begin{ques} Can the conclusion of {\em Theorem B} be obtained for a non-constant meromorphic function sharing a small function IM together with its $k$-th derivative counterpart? \end{ques}
Zhang \cite{14} extended {\it Theorem B} to meromorphic function and also studied the CM value sharing of a meromorphic function with its $k$-th derivative.\par
Meanwhile a new notion of scalings between CM and IM known as weighted sharing (\cite{6}), appeared in the uniqueness literature.

In 2004, Lahiri-Sarkar \cite{8} employed weighted value sharing method to improve the results of Zhang \cite{14}.
In 2005, Zhang \cite{15} further extended the results of Lahiri-Sarkar to a small function and proved the following result for IM sharing.
\begin{theoD}\cite{15} Let $f$ be a non-constant meromorphic function and $k(\geq 1)$ be integer. Also let $a\equiv a(z)$ ($\not\equiv 0,\infty$) be a meromorphic small function. Suppose that $f-a$ and $f^{(k)}-a$ share $0$ IM.
If \be\label{e1.1}4\ol N(r,\infty;f)+3N_{2}\left(r,0;f^{(k)}\right)+2\ol N\left(r,0;(f/a)^{'}\right)<(\lambda +o(1))\;T\left(r,f^{(k)}\right)\ee for $r\in I$, where $0<\lambda <1$ then $\frac{f^{(k)}-a}{f-a}=c$  for some constant $c\in\mathbb{C}/\{0\}$.\end{theoD}
We now recall the following two theorems due to Liu and Yang \cite{8b} in the direction of IM sharing related to {\it Theorem B} .
\begin{theoE} \cite{8b} Let $f$ be a non-constant meromorphic function. If $f$ and $f^{'}$ share $1$ IM and if \be\label{e1.2} \ol N(r,\infty;f)+\ol N\left(r,0;f^{'}\right)<(\lambda +o(1))\;T\left(r,f^{'}\right)\ee for $r\in I$, where $0<\lambda <\frac{1}{4}$ then $\frac{f^{'}-1}{f-1}\equiv c$  for some constant $c\in\mathbb{C}/\{0\} $.\end{theoE}
\begin{theoF} \cite{8b} Let $f$ be a non-constant meromorphic function and $k$ be a positive integer. If $f$ and $f^{(k)}$ share $1$ IM and \be\label{e1.3} (3k+6)\ol N(r,\infty;f)+5N(r,0;f)<(\lambda +o(1))\;T\left(r,f^{(k)}\right)\ee for $r\in I$, where $0<\lambda <1$ then $\frac{f^{(k)}-1}{f-1}\equiv c$  for some constant $c\in\mathbb{C}/\{0\}$.\end{theoF}

In 2008, improving the result of Zhang \cite{15}, Zhang and L$\ddot u$ \cite{16} further investigated the analogous problem of Br\"{u}ck conjecture for the $n$-th power of a meromorphic function sharing a small function with its $k$-th derivative and obtained the following theorem.
\begin{theoG}\cite{16} Let $f$ be a non-constant meromorphic function and $k(\geq 1)$ and $n(\geq 1)$ be integers. Also let $a\equiv a(z)$ ($\not\equiv 0,\infty$) be a meromorphic small function. Suppose that $f^{n}-a$ and $f^{(k)}-a$ share $0$ IM. If  \be\label{e1.4}4\ol N(r,\infty;f)+\ol N\left(r,0;f^{(k)}\right)+2N_{2}\left(r,0;f^{(k)}\right)+2\ol N\left(r,0;(f^{n}/a)^{'}\right)<(\lambda +o(1))\;T\left(r,f^{(k)}\right)\ee for $r\in I$, where $0<\lambda <1$ then $\frac{f^{(k)}-a}{f^{n}-a}=c$  for some constant $c\in\mathbb{C}/\{0\}$.\end{theoG}
At the end of \cite{16} the following question was raised by Zhang and L$\ddot u$ \cite{16}. \par
What will happen if $f^{n}$ and $[f^{(k)}]^{m}$ share a small function ?\par
In order to answer the above question, Liu \cite{8a} obtained the following result.
\begin{theoH}\cite{8a} Let $f$ be a non-constant meromorphic function and $k(\geq 1)$, $n(\geq 1)$ and $m(\geq 2)$ be integers. Also let $a\equiv a(z)$ ($\not\equiv 0,\infty$) be a meromorphic small function. Suppose that $f^{n}-a$ and $(f^{(k)})^{m}-a$ share $0$ IM. If \be\label{e1.7}\frac{4}{m}\ol N(r,\infty;f)+\frac{5}{m}\ol N\left(r,0;f^{(k)}\right)+\frac{2}{m}\ol N\left(r,0;(f^{n}/a)^{'}\right)<(\lambda +o(1))\;T\left(r,f^{(k)}\right)\ee for $r\in I$, where $0<\lambda <1$ then $\frac{(f^{(k)})^{m}-a}{f^{n}-a}=c$  for some constant $c\in\mathbb{C}/\{0\}$.\end{theoH}
Next we recall the following definitions.
\begin{defi} Let $n_{0j},n_{1j},\ldots,n_{kj}$ be non negative integers.\\
The expression $M_{j}[f]=(f)^{n_{0j}}(f^{(1)})^{n_{1j}}\ldots(f^{(k)})^{n_{kj}}$ is called a differential monomial generated by $f$ of degree $d(M_{j})=\sum\limits_{i=0}^{k}n_{ij}$ and weight
$\Gamma_{M_{j}}=\sum\limits_{i=0}^{k}(i+1)n_{ij}$.

The sum $P[f]=\sum\limits_{j=1}^{t}b_{j}M_{j}[f]$ is called a differential polynomial generated by $f$ of degree $\ol{d}(P)=max\{d(M_{j}):1\leq j\leq t\}$
and weight $\Gamma_{P}=max\{\Gamma_{M_{j}}:1\leq j\leq t\}$, where $T(r,b_{j})=S(r,f)$ for $j=1,2,\ldots,t$.

The numbers $\underline{d}(P)=min\{d(M_{j}):1\leq j\leq t\}$ and k(the highest order of the derivative of $f$ in $P[f]$ are called respectively the lower degree and order of $P[f]$.

$P[f]$ is said to be homogeneous if $\ol{d}(P)$=$\underline{d}(P)$.

$P[f]$ is called a Linear Differential Polynomial generated by $f$ if $\ol {d} (P)=1$. Otherwise $P[f]$ is called Non-linear Differential Polynomial. We also denote by $\mu =max\; \{\Gamma _{M_{j}}-d(M_{j}): 1\leq j\leq t\}=max\; \{ n_{1j}+2n_{2j}+\ldots+kn_{kj}: 1\leq j\leq t\}$. \end{defi}
As $(f^{(k)})^{m}$ is simply a special differential monomial in $f$, it will be interesting to investigate whether {\it Theorems D-H} can be extended up to differential polynomial generated by $f$.
In this direction recently Li and Yang \cite{8c} improved {\it Theorem D} in the following manner.
\begin{theoI}\cite{8c} Let $f$ be a non-constant meromorphic function $P[f]$ be a differential polynomial generated by $f$. Also let $a\equiv a(z)$ ($\not\equiv 0,\infty$) be a small meromorphic function. Suppose that $f-a$ and $P[f]-a$ share $0$ IM and $(t-1)\ol {d}(P)\leq \sum\limits_{j=1}^{t}d(M_{j})$. If  \bea\label{e1.8}&& 4\ol N(r,\infty;f)+3N_{2}\left(r,0;P[f]\right)+2\ol N\left(r,0;(f/a)^{'}\right)<(\lambda +o(1))\;T\left(r,P[f]\right)\eea for $r\in I$, where $0<\lambda <1$ then $\frac{P[f]-a}{f-a}=c$  for some constant $c\in\mathbb{C}/\{0\}$.\end{theoI}
So we see that {\it Theorem I} always holds for a monomial without any condition on its degree. But for general differential polynomial one can not eliminate the supposition $(t-1)\ol {d}(P)\leq \sum\limits_{j=1}^{t}d(M_{j})$ in the above theorem.
So whether in {\it Theorem I}, the condition over the degree can be removed, sharing notion can further be relaxed, (\ref{e1.8}) can further be weakened, are all open problems.\par
We also observe that the afterward research on Br\"{u}ck and its generalization, one setting among the sharing functions has been restricted to only various powers of $f$ not involving any other variants such as derivatives of $f$, where as the generalization have been made on the second setting.
This observation must motivate oneself to find the answer of the following question.
\begin{ques} Can Br\"{u}ck type conclusion be obtained when two different differential polynomials share a small functions IM or even under relaxed sharing notions ?\end{ques}
The main intention of the paper is to obtain the possible answer of the above question in such a way that it improves, unifies and generalizes all the {\it Theorems D-H}. Following theorem is the main result of the paper.
Henceforth by $b_{j}$, $j=1,2,\ldots, t$ and $c_{i}$ $i=1,2,\ldots, l$ we denote small functions in $f$ and we also suppose that $P[f]=\sum\limits_{j=1}^{t}b_{j}M_{j}[f]$ and $Q[f]=\sum\limits_{i=1}^{l}c_{i}M_{i}[f]$ be two differential polynomial generated by $f$.
\begin{theo}\label{t1} Let $f$ be a non-constant meromorphic function, $m(\geq 1)$ be a positive integer or infinity and $a\equiv a(z)$ ($\not\equiv 0,\infty$) be a small meromorphic function. Suppose that $P[f]$ and $Q[f]$ be two differential polynomial generated by $f$ such that $Q[f]$ contains at least one derivative. Suppose further that $\ol E_{m)}(a,P[f])=\ol E_{m)}(a,Q[f])$. If
\bea\label{e1.9}&& 4\ol N(r,\infty;f)+N_{2}\left(r,0;Q[f]\right)+2\ol N\left(r,0;Q[f]\right)+\ol N\left(r,0;(P[f]/a)^{'}\right)\\&&+\ol N\left(r,0;(P[f]/a)^{'}\mid (P[f]/a)\not =0\right)<(\lambda +o(1))\;T\left(r,Q[f]\right)\nonumber\eea for $r\in I$, where $0<\lambda <1$ then either $ \textbf a) \;\;\frac{Q[f]-a}{P[f]-a}=c,$ for some constant $c\in\mathbb{C}/\{0\}$ \\ or $ \textbf b)\;\; P[f]Q[f]-aQ[f](1+d)\equiv -da^{2},$ for a non-zero constant $d\in \mathbb{C}$.\\
In particular,\\ if i) $P[f]=b_{1}f^{n}+b_{2}f^{n-1}+b_{3}f^{n-2}+\ldots+b_{t-1}f$ or if
\\ii) $\underline {d}(Q) >2\ol {d}(P)-\underline {d}(P)$ and each monomial of $Q[f]$ contains a term involving a power of $f$, then the conclusion (b) does not hold.  \end{theo}
\begin{rem}\label{r1.1} Clearly in {\em Theorem {\ref{t1}}} when $m=\infty$ we have $P[f]-a$ and $Q[f]-a$ share $0$ IM where $P[f]=b_{1}f^{n}+b_{2}f^{n-1}+b_{3}f^{n-2}+\ldots+b_{t-1}f$ and we obtain the improved, extended and generalized version of {\em Theorem I} in the direction of {\em Question 1.1}.\end{rem}
Following five examples show that (\ref{e1.9}) is not necessary when (i) and (ii) of Theorem \ref{t1} occurs.
\begin{exm}\label{ex1.4}  Let $f(z)=\frac{e^{z}}{e^{z}+1}$. $P[f]=f^{2}$, $Q[f]=f-f^{'}$. Then clearly $P[f]$ and $Q[f]$ share $1$ CM and $\frac{Q[f]-1}{P[f]-1}=1$, but (\ref{e1.9}) is not satisfied.\end{exm}
\begin{exm}\label{ex1.5}  Let $f(z)=\frac{1}{e^{z}+1}$. $P[f]=f^{2}-f^{3}$, $Q[f]=-ff^{'}$. Then clearly $P[f]$ and $Q[f]$ share $1$ CM and $\frac{Q[f]-1}{P[f]-1}=1$, but (\ref{e1.9}) is not satisfied. \end{exm}
\begin{exm}\label{ex1.6}  Let $f(z)=\frac{e^{z}}{e^{z}+1}$. $P[f]=f-f^{'}$, $Q[f]=f^{2}-3f{f^{'}}^{3}+f^{3}{f^{'}}^{2}-ff^{'}f^{'''}+ff^{'}f^{''}$. Then clearly $P[f]$ and $Q[f]$ share $1$ CM and $\frac{Q[f]-1}{P[f]-1}=1$, but (\ref{e1.9}) is not satisfied.\end{exm}
\begin{exm}\label{ex1.7}  Let $f(z)=\frac{1}{e^{z}+1}$. $P[f]=(f^{'})^{2}-ff^{''}$, $Q[f]=2f{f^{'}}^{2}-f^{2}f^{''}$. Then clearly $P[f]$ and $Q[f]$ share $1$ CM and $\frac{Q[f]-1}{P[f]-1}=1$, but (\ref{e1.9}) is not satisfied. Here we note that $3=\underline{d}(Q)>2\ol{d}(P)-\underline{d}(P)=2$.\end{exm}
\begin{exm}\label{ex1.8}  Let $f(z)=\frac{1}{e^{z}+1}$. $P[f]={f^{'}}^{2}$, $Q[f]=ff^{''}-f^{2}f^{'}$. Then clearly $P[f]=Q[f]=\frac{e^{2z}}{{(e^{z}+1)}^{4}}$ share $\frac{1}{z}$ CM and $\frac{Q[f]-\frac{1}{z}}{P[f]-\frac{1}{z}}=1$, but (\ref{e1.9}) is not satisfied. \end{exm}

We now give the next five examples the first two of which show that both the conditions stated in (ii) are essential in order to obtain conclusion (a) in Theorem \ref{t1} for homogeneous differential polynomials $P[f]$ where as the rest three substantiate the same for non homogeneous differential polynomials.
\begin{exm}\label{ex1.9}  Let $f(z)=sin z$. $P[f]={f^{''}}^{2}-{f^{'}}^{2}+2if^{''}f^{'''}$, $Q[f]=f^{2}-2iff^{'}-{f^{'''}}^{2}$. Then clearly $P[f]=-e^{-2iz}$ and $Q[f]=-e^{2iz}$ share $1$ CM. Here $T(r,Q)=\frac{2r}{\pi }+O(1)$, (\ref{e1.9}) is satisfied, but $\frac{Q[f]-1}{P[f]-1}=e^{2iz}$, rather $P[f]Q[f]=1$. \end{exm}
\begin{exm}\label{ex1.10}  Let $f(z)=sin z$. $P[f]=3{f}^{2}+{f^{'}}^{2}-2iff^{'}$, $Q[f]=f^{2}-2iff^{'}-f^{2}$. Then clearly $P[f]=2-e^{2iz}$ and $Q[f]=e^{-2iz}$ share $1$ CM. Here (\ref{e1.9}) is satisfied, but $\frac{Q[f]-1}{P[f]-1}=e^{-2iz}$, rather $P[f]Q[f]-2Q[f]+1=0$. \end{exm}
\begin{exm}\label{ex1.11}  Let $f(z)=cos z$. $P[f]=f^{3}+3if^{'}{f^{'''}}^{2}+3{f^{'}}^{2}f^{''}-3if^{'}-i{f^{'''}}^{3}$, $Q[f]=3f^{''}-4{f^{''}}^{3}+3if^{2}f^{'}+i{f^{'''}}^{3}$. Then clearly $P[f]=e^{3iz}$ and $Q[f]=e^{-3iz}$ share $1$ CM. Here (\ref{e1.9}) is satisfied, but $\frac{Q[f]-1}{P[f]-1}=e^{-3iz}$ rather $P[f]Q[f]=1$. We also note that here $\ol{d}(P)\not =\underline{d}(P)$, $1=\underline{d}(Q)\not>2\ol{d}(P)-\underline{d}(P)=5$. \end{exm}
\begin{exm}\label{ex1.12}  Let $f(z)=cos z$. $P[f]=-2ff^{''}+{f^{'}}^{2}-f^{'}f^{'''}-f^{''}+if^{'''}$, $Q[f]=-f+if^{'''}$. Then clearly $P[f]=e^{iz}+2$ and $Q[f]=-e^{-iz}$ and so they share $1$ CM. Here (\ref{e1.9}) is satisfied, but $\frac{Q[f]-1}{P[f]-1}=-e^{-iz}$, rather $P[f]Q[f]-2Q[f]+1=0$. We also note that here $\ol{d}(P)\not =\underline{d}(P)$, $1=\underline{d}(Q)\not>2\ol{d}(P)-\underline{d}(P)=3$.\end{exm}
\begin{exm}\label{ex1.13}  Let $f(z)=cos z$. $P[f]=-f-if^{'}+(1+i){f^{'}}^{2}+(1+i){f^{''}}^{2}$, $Q[f]=if-f^{'''}$. Then clearly $P[f]=1+i-e^{-iz}$ and $Q[f]=ie^{iz}$ share both $i$ and $1$ CM. Here (\ref{e1.9}) is satisfied and $P[f]Q[f]-(1+i)Q[f]+i=0$. When we consider $i$ as the shared value then $\frac{Q[f]-i}{P[f]-i}=ie^{iz}$, on the other hand when we consider $1$ as the shared value then $\frac{Q[f]-1}{P[f]-1}=e^{iz}$. We also note that here $\ol{d}(P)\not =\underline{d}(P)$, $1=\underline{d}(Q)\not>2\ol{d}(P)-\underline{d}(P)=3$.\end{exm}

The following two examples show that in order to obtain conclusions (a) or (b) of {\it Theorem \ref{t1}}, (\ref{e1.9}) is essential.
\begin{exm}\label{ex1.14}  Let $f(z)=sin z$. $P[f]=if+f^{'}$, $Q[f]=2f^{'}-(f^{2}+{f^{'}}^{2})$. Then clearly $P[f]=e^{iz}$ and $Q[f]=e^{iz}+e^{-iz}-1$ share $1$ IM. Here neither of the conclusions of {\em Theorem \ref{t1}} is satisfied, nor (\ref{e1.9}) is satisfied. We note that $\frac{Q[f]-1}{P[f]-1}=\frac{(e^{iz}-1)}{e^{iz}}$ and $P[f]Q[f]-\lambda Q[f]$ is non-constant function for any $\lambda\in\mathbb{C}$. \end{exm}
\begin{exm}\label{ex1.15}  Let $f(z)=cos z$. $P[f]=f-if^{'}$, $Q[f]=2f-({f^{'}}^{2}+{f^{''}}^{2})$. Then clearly $P[f]=e^{iz}$ and $Q[f]=e^{iz}+e^{-iz}-1$ share $1$ IM. Here neither of the conclusions of {\em Theorem \ref{t1}} is satisfied, nor (\ref{e1.9}) is satisfied. We note that $\frac{Q[f]-1}{P[f]-1}=\frac{(e^{iz}-1)}{e^{iz}}$ and $P[f]Q[f]-\lambda Q[f]$ is non-constant function for any $\lambda\in\mathbb{C}$.\end{exm}

Though we use the standard notations and definitions of the value distribution theory available in \cite{4}, we explain some definitions and notations which are used in the paper.
\begin{defi}\cite{8}Let $p$ be a positive integer and $a\in\mathbb{C}\cup\{\infty\}$.\begin{enumerate}
\item[(i)] $N(r,a;f\mid \geq p)$ ($\ol N(r,a;f\mid \geq p)$)denotes the counting function (reduced counting function) of those $a$-points of $f$ whose multiplicities are not less than $p$.\item[(ii)]$N(r,a;f\mid \leq p)$ ($\ol N(r,a;f\mid \leq p)$)denotes the counting function (reduced counting function) of those $a$-points of $f$ whose multiplicities are not greater than $p$.\end{enumerate} \end{defi}
\begin{defi}\{6, cf.\cite {12}\} For $a\in\mathbb{C}\cup\{\infty\}$ and a positive integer $p$ we denote by $N_{p}(r,a;f)$ the sum $\ol N(r,a;f)+\ol N(r,a;f\mid\geq 2)+\ldots\ol N(r,a;f\mid\geq p)$. Clearly $N_{1}(r,a;f)=\ol N(r,a;f)$. \end{defi}
\begin{defi} Let $k$ be a positive integer and  for $a\in\mathbb{C}-\{0\}$, $\ol E_{k)}(a;f)=\ol E_{k)}(a;g)$. Let $z_{0}$ be a zero of $f(z)-a$ of multiplicity $p$ and a zero of $g(z)-a$ of multiplicity $q$. We denote by $\ol N_{L}(r,a;f)$ the counting function of those $a$-points of $f$ and $g$ where $p>q\geq 1$, by $\ol N_{f>s}(r,a;g)$ ($\ol N_{g>s}(r,a;f)$) the counting functions of those $a$-points of $f$ and $g$ for which $p>q=s$($q>p=s$), by $N^{1)}_{E}(r,a;f)$ the counting function of those $a$-points of $f$ and $g$ where $p=q=1$ and by $\ol N^{(2}_{E}(r,a;f)$ the counting function of those $a$-points of $f$ and $g$ where $p=q\geq 2$, each point in these counting functions is counted only once. In the same way we can define $\ol N_{L}(r,a;g),\; N^{1)}_{E}(r,a;g),\; \ol N^{(2}_{E}(r,a;g).$
We denote by $\ol N_{f\geq k+1}(r,a;f\mid\; g\not=a)$ ($\ol N_{g\geq k+1}(r,a;g\mid\; f\not=a)$) the reduced counting functions of those $a$-points of $f$ and $g$ for which $p\geq k+1$ and $q=0$ ($q\geq k+1$ and $p=0$).\end{defi}
\begin{defi}\cite{7} Let $a,b \in\mathbb{C}\;\cup\{\infty\}$. We denote by $N(r,a;f\mid\; g\neq b)$ the counting function of those $a$-points of $f$, counted according to multiplicity, which are not the $b$-points of $g$.\end{defi}
\begin{defi}\cite{6} Let $f$, $g$ share a value $a$ IM. We denote by $\ol N_{*}(r,a;f,g)$ the reduced counting function of those $a$-points of $f$ whose multiplicities differ from the multiplicities of the corresponding $a$-points of $g$.

Clearly $\ol N_{*}(r,a;f,g)\equiv\ol N_{*}(r,a;g,f)$ and $\ol N_{*}(r,a;f,g)=\ol N_{L}(r,a;f)+\ol N_{L}(r,a;g)$.\end{defi}

\section{Lemmas} In this section we present some lemmas which will be needed in the sequel. Let $F$, $G$ be two non-constant meromorphic functions. Henceforth we shall denote by $H$ the following function. \be\label{e2.1}H=\left(\frac{\;\;F^{''}}{F^{'}}-\frac{2F^{'}}{F-1}\right)-\left(\frac{\;\;G^{''}}{G^{'}}-\frac{2G^{'}}{G-1}\right).\ee
\begin{lem}\label{l2.1a} Let $\ol E_{m)}(1;F)=\ol E_{m)}(1;G)$; $F$, $G$ share $\infty$ IM and $H\not\equiv 0$. Then \beas && N(r,\infty;H)\\&\leq&\ol N(r,0;F\mid\geq 2)+\ol N(r,0;G\mid\geq 2)+\ol N_{*}(r,\infty;F,G)+\ol N_{F\geq m+1}(r,1;F\mid\;G\not=1)\\& &+\ol N_{G\geq m+1}(r,1;G\mid\;F\not=1)+\ol N_{L}(r,1;F)+\ol N_{L}(r,1;G)+\ol N_{0}(r,0;F^{'})+\ol N_{0}(r,0;G^{'}),\eeas where $\ol N_{0}(r,0;F^{'})$ is the reduced counting function of those zeros of $F^{'}$ which are not the zeros of $F(F-1)$ and $\ol N_{0}(r,0;G^{'})$ is similarly defined.\end{lem}
\begin{proof} We can easily verify that possible poles of $H$ occur at (i) multiple zeros of $F$ and $G$, (ii) poles of $F$ and $G$ with different multiplicities, (iii) the common zeros of $F-1$ and $G-1$ with different multiplicities, (iii) zeros of $F-1$ ($G-1$) which are not the zeros of $G-1$ ($F-1$), (iv) those $1$-points of $F$ $(G)$  which are not the $1$-points of $G$ $(F)$, (v) zeros of $F^{'}$ which are not the zeros of $F(F-1)$, (vi) zeros of $G^{'}$ which are not zeros of $G(G-1)$. Since $H$ has simple pole the lemma follows from above.\end{proof}
\begin{lem}\label{l2.1}\cite{15} Let $f$ be a non-constant meromorphic function and $k$ be a positive integer, then $$N_{p}(r,0;f^{(k)})\leq N_{p+k}(r,0;f)+k\ol N(r,\infty;f)+S(r,f).$$\end{lem}
\begin{lem}\label{l2.2}\cite{7a} If $N(r,0;f^{(k)}\mid f\not=0)$ denotes the counting function of those zeros of  $f^{(k)}$ which are not the zeros of $f$, where a zero of $f^{(k)}$ is counted according to its multiplicity then $$N(r,0;f^{(k)}\mid f\not=0)\leq k\ol N(r,\infty;f)+N(r,0;f\mid <k)+k\ol N(r,0;f\mid\geq k)+S(r,f).$$\end{lem}
\begin{lem}\label{l2.3}\cite{9} Let $f$ be a non-constant meromorphic function and let \[R(f)=\frac{\sum\limits _{k=0}^{n} a_{k}f^{k}}{\sum \limits_{j=0}^{m} b_{j}f^{j}}\] be an irreducible rational function in $f$ with constant coefficients $\{a_{k}\}$ and $\{b_{j}\}$ where $a_{n}\not=0$ and $b_{m}\not=0$. Then $$T(r,R(f))=dT(r,f)+S(r,f),$$ where $d=\max\{n,m\}$.\end{lem}
\begin{lem} \label{l2.4}\cite{3a} Let $f$ be a meromorphic function and $P[f]$ be a differential polynomial. Then
$$ m\left(r,\frac{P[f]}{f^{\ol {d}(P)}}\right)\leq (\ol {d}(P)-\underline {d}(P)) m\left(r,\frac{1}{f}\right)+S(r,f).$$\end{lem}
\begin{lem} \label{l2.5} Let $f$ be a meromorphic function and $P[f]$ be a differential polynomial. Then we have
\beas N\left(r,\infty;\frac{P[f]}{f^{\ol {d}(P)}}\right)&\leq& (\Gamma _{P}-\ol {d}(P))\;\ol N(r,\infty;f)+(\ol {d}(P)-\underline {d} (P))\; N(r,0;f\mid\geq k+1)\\&&+\mu \ol N(r,0;f\mid\geq k+1)+\ol {d}(P) N(r,0;f\mid\leq k)+S(r,f).\eeas \end{lem}
\begin{proof} Let $z_{0}$ be a pole of $f$ of order $r$, such that $b_{j}(z_{0})\not=0,\infty: 1\leq j\leq t$. Then it would be a pole of $P[f]$ of order at most $r\ol {d}(P)+\Gamma _{P}-\ol {d}(P)$. Since $z_{0}$ is a pole of $f^{\ol {d}(P)}$ of order $r\ol {d}(P)$, it follows that $z_{0}$ would be a pole of $\frac{P[f]}{f^{\ol {d}(P)}}$ of order at most $\Gamma _{P}-\ol {d}(P)$.
Next suppose $z_{1}$ is a zero of $f$ of order $s(>k)$, such that $b_{j}(z_{1})\not=0,\infty: 1\leq j\leq t$. Clearly it would be a zero of $M_{j}(f)$ of order $s.n_{0j}+(s-1)n_{1j}+\ldots+(s-k)n_{kj}= s.d(M_{j})-(\Gamma _{M_{j}}-d(M_{j}))$.
 Hence $z_{1}$ be a pole of $\frac{M_{j}[f]}{f^{\ol {d}(P)}}$ of order $$s.\ol {d}(P)-s.d(M_{j})+(\Gamma _{M_{j}}-d(M_{j}))=s(\ol {d}(P)-d(M_{j}))+(\Gamma _{M_{j}}-d(M_{j})).$$ So $z_{1}$ would be a pole of $\frac{P[f]}{f^{\ol {d}(P)}}$ of order at most $$\text{max} \{s(\ol {d}(P)-d(M_{j}))+(\Gamma _{M_{j}}-d(M_{j})): 1\leq j\leq t)\}=s(\ol {d}(P)-\underline{d}(P))+\mu .$$
If $z_{1}$ is a zero of $f$ of order $s\leq k$, such that $b_{j}(z_{1})\not=0,\infty: 1\leq j\leq t$ then it would be a pole of  $\frac{P[f]}{f^{\ol {d}(P)}}$ of order $s\ol {d}(P)$. Since the poles of $\frac{P[f]}{f^{\ol {d}(P)}}$ comes from the poles or zeros of $f$ and poles or zeros of $b_{j}(z)$'s only, it follows that
\beas  N\left(r,\infty;\frac{P[f]}{f^{\ol {d}(P)}}\right)&\leq& (\Gamma _{P}-\ol {d}(P))\;\ol N(r,\infty;f)+(\ol {d}(P)-\underline {d} (P))\; N(r,0;f\mid \geq k+1)\\&&+\mu \;\ol N(r,0;f\mid\geq k+1)+\ol {d}(P) N(r,0;f\mid\leq k)+S(r,f).\eeas \end{proof}
\begin{lem}\label{l2.6} \cite{3b} Let $P[f]$ be a differential polynomial. Then $$T(r,P[f])\leq\Gamma_{P}T(r,f)+S(r,f).$$\end{lem}
\begin{lem} \label{l2.7} Let $f$ be a non-constant meromorphic function and $P[f]$ be a differential polynomial. Then $S(r,P[f])$ can be replaced by $S(r,f)$.\end{lem}
\begin{proof} From {\it Lemma \ref{l2.6}} it is clear that $T(r,P[f])=O(T(r,f))$ and so the lemma follows. \end{proof}
\begin{lem} \label{l2.8} Let $f$ be a non-constant meromorphic function and $P[f]$, $Q[f]$ be two differential polynomials. Then \beas && N(r,0;P[f])\\&\leq& \frac{\ol {d}(P)-\underline {d}(P)}{\underline {d}(Q)} m\left(r,\frac{1}{Q[f]}\right)+(\Gamma _{P}-\ol {d}(P))\;\ol N(r,\infty;f)+(\ol {d}(P)-\underline {d} (P))\; N(r,0;f\mid\geq k+1)\\&&+\mu \ol N(r,0;f\mid\geq k+1)+\ol {d}(P) N(r,0;f\mid\leq k)+S(r,f).\eeas \end{lem}
\begin{proof} For a fixed value of $r$, let $E_{1}=\{\theta \in [0,2\pi]: \left|f(re^{i\theta })\right|\leq 1 \}$ and  $E_{2}$ be its complement. Since by definition $$\sum\limits_{i=0}^{k} n_{ij}\geq \underline {d}(Q),$$ for every $j=1,2,\ldots,l$, it follows that on $E_{1}$
$$\left|\frac{Q[f]}{f^{\underline{d}(Q)}}\right| \leq \sum\limits_{j=1}^{l}\left| c_{j}(z)\right| \prod\limits_{i=1}^{k}\left| \frac{f^{(i)}}{f} \right|^{n_{ij}} \left|f\right|^{^{\sum\limits_{i=0}^{k}n_{}{ij}-\underline{d}(Q)}}\leq \sum\limits_{j=1}^{l}\left| c_{j}(z)\right| \prod\limits_{i=1}^{k}\left| \frac{f^{(i)}}{f} \right|^{n_{ij}}.$$
Also we note that $$\frac{1}{f^{\underline{d}(Q)}}=\frac{Q[f]}{f^{\underline{d}(Q)}}\;\frac{1}{Q[f]}.$$
Since on $E_{2}$, $\frac{1}{\left|f(z)\right|}< 1$, we have
\beas && \underline{d}(Q)m\left(r,\frac{1}{f}\right) \\&=& \frac{1}{2\pi}\int\limits_{E_{1}}\log ^{+}\frac{1}{\left|f(re^{i\theta })\right|^{\underline{d}(Q)}}d\theta+\frac{1}{2\pi}\int\limits_{E_{2}}\log ^{+}\frac{1}{\left|f(re^{i\theta })\right|^{\underline{d}(Q)}}d\theta\\ &\leq&\frac{1}{2\pi}\sum\limits_{j=1}^{l}\left[\int\limits_{E_{1}}\log ^{+}\left|c_{j}(z)\right|d\theta+\sum\limits_{i=1}^{k}\int\limits_{E_{1}}\log ^{+}\left|\frac{f^{(i)}}{f}\right|^{n_{ij}}d\theta\right]+\frac{1}{2\pi}\int\limits_{E_{1}}\log ^{+}\left|\frac{1}{Q[f(re^{i\theta})]}\right|d\theta\\&\leq &\frac{1}{2\pi}\int\limits_{0}^{2\pi }\log ^{+}\left|\frac{1}{Q[f(re^{i\theta})]}\right|d\theta+S(r,f)=m\left(r,\frac{1}{Q[f]}\right)+S(r,f).\eeas
So using {\it Lemmas {\ref{l2.4}}}, {\it \ref{l2.5}} and the first fundamental theorem we get \beas && N(r,0;P[f])\\&\leq& N\left(r,\infty;\frac{f^{\overline{d}(P)}}{P[f]}\right)+\overline{d}(P)N(r,0;f)\\&\leq& m\left(r,\frac{P[f]}{f^{\overline{d}(P)}}\right)+N\left(r,\infty;\frac{P[f]}{f^{\overline{d}(P)}}\right)+\overline{d}(P)N(r,0;f)+S(r,f)\\&\leq& (\ol {d}(P)-\underline {d}(P))m\left(r,\frac{1}{f}\right)+(\Gamma _{P}-\ol {d}(P))\;\ol N(r,\infty;f)+(\ol {d}(P)-\underline {d} (P))\; N(r,0;f\mid\geq k+1)\\&&+\mu \ol N(r,0;f\mid\geq k+1)+\ol {d}(P) N(r,0;f\mid\leq k)+S(r,f)\\&\leq& \frac{(\ol {d}(P)-\underline {d}(P))}{\underline {d}(Q)}m\left(r,\frac{1}{Q[f]}\right)+(\Gamma _{P}-\ol {d}(P))\;\ol N(r,\infty;f)+(\ol {d}(P)-\underline {d} (P))\; N(r,0;f\mid\geq k+1)\\&&+\mu \ol N(r,0;f\mid\geq k+1)+\ol {d}(P) N(r,0;f\mid\leq k)+S(r,f). \eeas\end{proof}
\section {Proof of the theorem}
\begin{proof} [Proof of Theorem \ref{t1}] Let $F=\frac{P[f]}{a}$ and $G=\frac{Q[f]}{a}$. Then $F-1=\frac{P[f]-a}{a}$, $G-1=\frac{Q[f]-a}{a}$. Since $\ol E_{m)}(a,P[f])=\ol E_{m)}(a,Q[f])$, it follows that $\ol E_{m)}(1,F)=\ol E_{m)}(1,G)$ except the zeros and poles of $a(z)$. Now we consider the following cases.\\
{\bf Case 1} Let $H\not\equiv 0$.\par
Let $z_{0}$ be a simple zero of $F-1$. Then by a simple calculation we see that $z_{0}$ is a zero of $H$ and hence \be\label{e3.2}N^{1)}_{E}(r,1;F)=N^{1)}_{E}(r,1;G)\leq N(r,0;H)\leq N(r,\infty;H)+S(r,F)\ee
Using (\ref{e3.2}), {\it Lemmas \ref{l2.1a}}, {\it \ref{l2.7}} and noting that $\ol N(r,\infty;F)=\ol N(r,\infty;G)+S(r,f)=\ol N(r,\infty;f)+S(r,f)$ and $\ol N_{F>1}(r,1;G)+\ol N(r,1;G\mid\geq 2)=\ol N_{E}^{(2}(r,1;G)+\ol N_{L}(r,1;G)+\ol N_{L}(r,1;F)+\ol N_{G\geq m+1}(r,1;G\mid\;F\not=1)+S(r,f)$, we get from the second fundamental theorem that
\bea\label{e3.3}&& T(r,G)\\&\leq& \ol N(r,\infty;G)+\ol N(r,0;G)+N^{1)}_{E}(r,1;G)+\ol N_{F>1}(r,1;G)+\ol N(r,1;G\mid\geq 2)\nonumber\\&&-N_{0}(r,0;G^{'})+S(r,G)\nonumber\\&\leq& 2\ol N(r,\infty;F)+\ol N(r,0;G)+\ol N(r,0;G\mid\geq 2)+\ol N(r,0;F\mid\geq 2)+2\ol N_{L}(r,1;F)\nonumber\\& &+2\ol N_{L}(r,1;G)+\ol N_{F\geq m+1}(r,1;F\mid\;G\not=1)+2\ol N_{G\geq m+1}(r,1;G\mid\;F\not=1)\nonumber\\& &+\ol N_{E}^{(2}(r,1;G)+\ol N_{0}(r,0;F^{'})+S(r,f).\nonumber\eea
Using {\it Lemmas \ref{l2.1}}, {\it \ref{l2.2}} we see that
\bea \label{e3.4}&& \ol N(r,0;G\mid \geq 2)+2\ol N_{G\geq m+1}(r,1;G\mid\;F\not=1)+2\ol N_{L}(r,1;G)+\ol N_{E}^{(2}(r,1;G)\\&\leq& \ol N(r,0;G^{'}\mid G\not=0)+\ol N(r,0;G^{'})+S(r,f)\nonumber\\&\leq& 2\ol N(r,\infty;f)+\ol N(r,0;Q[f])+N_{2}(r,0;Q[f])+S(r,f)\nonumber\eea
and \bea \label{e3.5}&& \ol N(r,0;F\mid\geq 2)+\ol N_{F\geq m+1}(r,1;F\mid\;G\not=1)+2\ol N_{L}(r,1;F)+\ol N_{0}(r,0;F^{'})\\&\leq& \ol N(r,0;F^{'}\mid F\not=0)+\ol N(r,0;F^{'})+S(r,f)\nonumber\\&\leq& \ol N(r,0;(P[f]/a)^{'}\mid (P[f]/a)\not=0)+\ol N(r,0;(P[f]/a)^{'})+S(r,f)\nonumber\eea
Using (\ref{e3.4}) and (\ref{e3.5}) in (\ref{e3.2}) we have \beas T\left(r,Q[f]\right)&\leq& 4\ol N(r,\infty;f)+2\ol N\left(r,0;Q[f]\right)+N_{2}\left(r,0;Q[f]\right)+ \ol N\left(r,0;(P[f]/a)^{'}\right)\\&&+\ol N\left(r,0;(P[f]/a)^{'}\mid (P[f]/a)\not=0\right)+S(r,f).\eeas
This contradicts (\ref{e1.9}).\\
 {\bf Case 2} Let $H\equiv 0$.\\ Suppose $F=P[f]/a$ and $G=Q[f]/a$. On integration we get from  \be\label{e3.8}\frac{1}{F-1}\equiv\frac{C}{G-1}+D,\ee where $C$, $D$ are constants and $C\not=0$. From (\ref{e3.8}) it is clear that $F$ and $G$ share $1$ CM.
We first assume that $D\not=0$. Then by (\ref{e3.8}) we get \be\label{e3.8a}\ol N(r,\infty;f)=S(r,f).\ee \par
Clearly $\ol N(r,\infty;G)=\ol N(r,\infty;f)+S(r,f)=S(r,f)$.\par
From (\ref{e3.8}) we get \be\label{e3.9}\frac{1}{F-1}=\frac{D\left(G-1+\frac{C}{D}\right)}{G-1}\ee
Clearly from (\ref{e3.9}) we have \be\label{e3.10}\ol N\left(r,1-\frac{C}{D};G\right)=\ol N(r,\infty;F)=\ol N(r,\infty;G)=S(r,f).\ee

If $\frac{C}{D}\not=1$, by the second fundamental theorem, {\it Lemma \ref{l2.7}} and (\ref{e3.10}) we have \beas T(r,G)&\leq& \ol N(r,\infty;G)+\ol N(r,0;G)+\ol N\left(r,1-\frac{C}{D};G\right)+S(r,G)\\&\leq&\ol N(r,0;G)+S(r,f)\leq N_{2}(r,0;G)+S(r,f)\\&\leq& T(r,G)+S(r,f).\eeas
So $T(r,G)= N_{2}(r,0;G)+S(r,f)$ that is, $T\left(r,Q[f]\right)=N_{2}\left(r,0;Q[f]\right)+S(r,f)$, which contradicts (\ref{e1.9}).\par
If $\frac{C}{D}=1$ we get from (\ref{e3.8}) \bea\label{e3.12} \left(F-1-\frac{1}{C}\right)G\equiv -\frac{1}{C}.\eea
i.e., $$P[f]Q[f]-aQ(1+d)\equiv -da^{2},$$ for a non zero constant $d=\frac{1}{C}\in \mathbb{C}$.
From (\ref{e3.12}) it follows that \be \label{e3.13} N(r,0;f\mid \geq k+1)\leq N(r,0;Q[f])\leq N(r,0;G)\leq N(r,0;a)=S(r,f).\ee
\par When $P[f]=b_{1}f^{n}+b_{2}f^{n-1}+b_{3}f^{n-2}+\ldots+b_{t-1}f$, we see from (\ref{e3.12}) that $$\frac{1}{f^{\ol {d}(Q)}\left(P[f]-(1+1/C)a\right)}\equiv -\;\frac{C}{a^{2}}\;\;\frac{Q[f]}{f^{\ol {d}(Q)}}.$$
Hence by the first fundamental theorem, (\ref{e3.8a}), (\ref{e3.13}), {\it Lemmas \ref{l2.3}}, {\it \ref{l2.4}} and {\it \ref{l2.5}} we get that
\bea\label{e3.14} &&(n+\ol {d}(Q))T(r,f)\\&=&T\left(r,f^{\ol {d}(Q)}(P[f]-(1+\frac{1}{C})a)\right)+S(r,f)\nonumber\\&=& T\left(r,\frac{1}{f^{\ol {d}(Q)}(P[f]-(1+\frac{1}{C})a)}\right)+S(r,f)\nonumber\\&=&T\left(r,\frac{Q[f]}{f^{\ol {d}(Q)}}\right)+S(r,f)\nonumber\\&\leq& m\left(r,\frac{Q[f]}{f^{\ol {d}(Q)}}\right)+N\left(r,\frac{Q[f]}{f^{\ol {d}(Q)}}\right)+S(r,f)
\nonumber\\&\leq& (\ol {d}(Q)-\underline {d}(Q)) \left[T(r,f)-\{N(r,0;f\mid\leq k)+N(r,0;f\mid \geq k+1)\}\right]+(\ol {d}(Q)-\underline {d}(Q))\nonumber\\&& N(r,0;f\mid\geq k+1)+\mu \;\ol N(r,0;f\mid\geq k+1)+\ol {d}(Q)N(r,0;f\leq k)+S(r,f)\nonumber\\&\leq& (\ol {d}(Q)-\underline {d}(Q)) T(r,f)+\underline {d}(Q) N(r,0;f\mid\leq k)+S(r,f).\nonumber\eea
From (\ref {e3.14}) it follows that \beas nT(r,f)\leq S(r,f),\eeas which is absurd.\\
If $P[f]$ is a differential polynomial then we consider the following two subcases.\\
{\bf Subcase 2.1.}\par
If $C=-1$ then from (\ref{e3.8}) we get $FG\equiv 1$, i.e., $P[f] Q[f]\equiv a^{2}$. It is clear that $\ol N(r,\infty;P[f])=\ol N(r,\infty;Q[f])=S(r,f)$.  \par
First we observe that since each monomial of $Q[f]$ contains a term involving a power of $f$, we have $N(r,0;f)=S(r,f)$. So from the first fundamental theorem, {\it Lemma \ref{l2.4}} and noting that $m\left(r,\frac{1}{f}\right)\leq \frac{1}{\underline {d}(Q)} m(r,\frac{1}{Q[f]}))$ we have
\beas T(r,Q[f])&\leq& T(r,P[f])+S(r,f)\\&\leq& m(r,\frac{P[f]}{f^{\ol {d}(P)}})+\ol {d}(P) m(r,f)+S(r,f)\\&\leq& (\ol {d}(P)-\underline {d}(P))m(r,\frac{1}{f})+\ol {d}(P) m(r,f)+S(r,f)\\&\leq&\frac{(\ol {d}(P)-\underline {d}(P))}{\underline {d}(Q)} m(r,\frac{1}{Q[f]})+\ol {d}(P)\{m(r,\frac{1}{f})+N(r,0;f)\}+S(r,f)\\&\leq& \frac{(\ol {d}(P)-\underline {d}(P))}{\underline {d}(Q)} m(r,\frac{1}{Q[f]})+\frac{\ol {d}(P)}{\underline {d}(Q)}m(r,\frac{1}{Q[f]})+S(r,f),
\eeas which is a contradiction as $\underline  {d}(Q) >2\ol {d}(P)-\underline {d}(P)$.\\
{\bf Subcase 2.2.}\par
Next we assume $C\not =-1$. \par
Then from (\ref{e3.12}) we have $$\ol N(r,1+\frac{1}{C};F)=\ol N(r,\infty;G)=S(r,f).$$\\
So again noticing the fact that each monomial of $Q[f]$ contains a term involving a power of $f$, by the second fundamental theorem, {\it Lemma \ref{l2.8}} we get
\bea\label{e3.15}&& T(r,P[f])\\&\leq& \ol N(r,\infty;F)+\ol N(r,0;F)+\ol N(r,1+\frac{1}{C};F)+S(r,f)\nonumber\\&\leq& N(r,0;P[f])+S(r,f)\nonumber\\&\leq &\frac{\ol {d}(P)-\underline {d}(P)}{\underline {d}(Q)} T(r,P[f])+S(r,f),\nonumber\eea
i.e., \be\label{e3.16} \frac{\underline{d}(Q)+\underline{d}(P)-\ol {d}(P)}{\underline{d}(Q)}T(r,P[f])\leq S(r,f).\ee  Since by the given condition $\underline{d}(Q)>2\ol {d}(P)-\underline{d}(P)>\ol {d}(P)-\underline{d}(P)$ (\ref{e3.16}) leads to a contradiction.\\
 Hence $D=0$ and so $\frac{G-1}{F-1}=C$ or $\frac{Q[f]-a}{P[f]-a}=C$. This proves the theorem. \end{proof}
\section{Concluding Remark and an Open Question}
From the statement of {\it Theorem \ref{t1}} one can see that when (ii) happens one can not obtain the conclusion of Br\"{u}ck conjecture as a special case.
We also see from (\ref{e3.8a}) that if $\ol N(r,\infty;f)\not=S(r,f)$ then conclusion of Br\"{u}ck conjecture is satisfied for any two arbitrary differential polynomials $P[f]$ and $Q[f]$ where $Q[f]$ contains at least one derivative. The problem arises for those class of meromorphic functions whose poles are relatively small in numbers such as entire functions and thus poles have a vital contributions in this perspective.
We point out that the counter examples (\ref {ex1.9})-(\ref {ex1.13}), which demonstrate the indispensability of the conditions in (ii), have also been formed for entire functions.
So the following question still remain open for further investigations. \\
{\it Can Br\"{u}ck type conclusion be solely obtained for two arbitrary differential polynomials $P[f]$ and $Q[f]$ generated by the class of meromorphic functions containing relatively small number of poles sharing a small function $a\equiv a(z)$ ($\not\equiv 0,\infty$) IM  }?
\begin{center} {\bf Acknowledgement} \end{center}
This research work is supported by the Council Of Scientific and Industrial Research, Extramural
Research Division, CSIR Complex, Pusa, New Delhi-110012, India, under the sanction project no. 25(0229)/14/EMR-II.

\end{document}